\documentclass[11pt]{amsart}

\usepackage{amscd,amsthm,amsfonts,amssymb,esint}
\usepackage{amsmath}
\usepackage[colorlinks=true]{hyperref}
\usepackage{fullpage}
\usepackage[all]{xy}
\usepackage{graphicx}

    \newcommand{\BC}{{\mathbb {C}}} 
    \newcommand{\BE}{{\mathbb {E}}} 
     \newcommand{\BH}{{\mathbb {H}}}

    \newcommand{\BQ}{{\mathbb {Q}}} \newcommand{\BR}{{\mathbb {R}}}

     \newcommand{\BZ}{{\mathbb {Z}}}

    \newcommand{\CM}{{\mathcal {M}}} 
    \newcommand{\CO}{{\mathcal {O}}}

     \newcommand{\GL}{{\mathrm{GL}}}

    \renewcommand{\Im}{{\mathrm{Im}}}

    \newcommand{\Isom}{{\mathrm{Isom}}}
     
    \newcommand{\Ker}{{\mathrm{Ker}}}

    \newcommand{\lcm}{{\mathrm{lcm}}}

    \newcommand{\PSL}{{\mathrm{PSL}}}
    \renewcommand{\mod}{\ \mathrm{mod}\ }\renewcommand{\Re}{{\mathrm{Re}}}

    \newcommand{\sign}{{\mathrm{sign}}}
    \newcommand{\SL}{{\mathrm{SL}}}

    \newcommand{\tr}{{\mathrm{tr}}}

\newcommand{\matrixx}[4]{\begin{pmatrix}
#1 & #2 \\ #3 & #4
\end{pmatrix} }        

    \newcommand{\wt}{\widetilde}

    \newcommand{\ov}{\overline}

    \newcommand{\sk}{\medskip}
    \newcommand{\lra}{\longrightarrow}
    \newcommand{\ra}{\rightarrow}

    \newcommand{\s}{\sk\noindent}

    \theoremstyle{plain}
    \newtheorem{thm}{Theorem}[section] \newtheorem{cor}[thm]{Corollary}
    \newtheorem{lem}[thm]{Lemma}  \newtheorem{prop}[thm]{Proposition}
    \newtheorem {conj}[thm]{Conjecture} 
    \newtheorem{lem-defn}[thm]{Lemmma-Definition}

\theoremstyle{remark} 
\theoremstyle{remark} 
\theoremstyle{remark} 

    \numberwithin{equation}{section}

        \newcommand{\Sol}{\mathrm{Sol}}
         \newcommand{\Nil}{\mathrm{Nil}}

\title{Achirality of Sol 3-Manifolds,  Stevenhagen Conjecture \\ and Shimizu's L-series}

\author{Ye Tian}
\address{Academy of Mathematics and System Science,  Morningside Center of Mathematics, 
Chinese Academy of Sciences,
Beijing 100190}
\email{ytian@math.ac.cn}
\author{Shicheng Wang}
\address{Department of Mathematical Sciences, Peking University, Beijing 100871, CHINA}
\email{wangsc@math.pku.edu.cn}
\author{Zhongzi Wang}
\address{Department of Mathematical Sciences, Peking University, Beijing 100871 CHINA}
\email{wangzz22@stu.pku.edu.cn}

\date{\today}

\begin{document}



\begin{abstract} 
A closed orientable manifold  is {\em achiral} if it admits an orientation reversing homeomorphism.
A  commensurable class of closed manifolds is achiral if it contains an achiral element,
or equivalently, each manifold in $\CM$ has an achiral finite cover. 
  Each commensurable class containing  non-orientable elements must be achiral.  
   It is natural to wonder  how many 
 commensurable classes are achiral and how many achiral classes have non-orientable elements.

We study this problem for Sol 3-manifolds. 
Each commensurable class $\CM$ of Sol 3-manifold has a complete topological invariant $D_{\CM}$, the discriminant of $\CM$. 
Our main result  is:


(1) Among all commensurable classes of Sol 3-manifolds, 
there are infinitely many achiral classes; however ordered by discriminants, the  density of achiral commensurable classes is 0.

 
(2) Among all achiral commensurable classes of Sol 3-manifolds, ordered by discriminants, the density of classes containing non-orientable elements is $1-\rho$,
    where
$$\rho:=\prod_{j=1}^\infty  \left(1+2^{-j}\right)^{-1} = 0.41942\cdots.$$

\end{abstract}

\maketitle

\tableofcontents

 \section{Introduction}
In this paper, we assume that all manifolds $M$ are closed. A closed orientable manifold is called {\em achiral},  if it admits an orientation reversing homeomorphism, and is called virtually achiral if it has an achiral finite cover.
We call two manifolds  $M_1$ and $M_2$ commensurable, if they  have  a common finite cover. A commensurable class of closed manifolds is called achiral if it contains an achiral element. A commensurable class  $\CM$ is achiral is equivalent to that each manifold in $\CM$ is virtually achiral (Lemma \ref{virtual}). Since each non-orientable manifold has an orientable double cover which is achiral,   each commensurable class containing  non-orientable elements must be achiral. 

The achirality is fundamental in topology for the classification of oriented manifolds, specially in dimension 3, where the concept achirality is originated. The study of various virtual properties, say  virtually Haken, virtually positive volume, virtual dominations, are important and active topics in 3-manifolds,  see a survey \cite{LS}. 



It will be natural to wonder the following questions: In cartain range,

{\it 1. What is the density  of achiral manifolds  among all?  }

{\it 2. What is the density of  manifolds have an  achiral finite cover among all? or what is the density of achiral classes among all commensurable classes?}

{\it 3.  What is the density of classes containing non-orientable element among all achiral classes?}

Question 1 is for manifolds and Question 2 and Question 3 are for commensurable classes.   
Note to study those questions, one need first to order the manifolds or classes in the range considered, which is subtle, and often
difficult.

In dimension 1 and 2, the answers to  those questions are direct: There is only one closed 1-manifold, the circle, which is achiral.
All orientable 2-manifolds are achiral. 
Denote $\BH^n$, $\BE^n$, $S^n$ the $n$-dimensional hyperbolic, Euclidean, and spherical geometries, respectively.  In dimension two case, every closed surface supports one of the three geometries:
$S^2, \BE^2, \BH^2.$
Each geometry has only one commensurable class which  also contains non-orientable elements.

Once we come to dimension 3, those questions become highly non-trivial. 
Recall that Thurston's eight geometries among 3-manifolds are \cite{Th} \cite{Sc},
 $$  S^3,\quad  \BE^3,\quad   S^2\times \BE^1, \quad \BH^2\times \BE^1, \quad  \Nil, \quad \wt{\PSL(2,\mathbb{R})},\quad  \BH^3,\quad  \Sol.$$

It is neither  meaningful  nor possible for us to order all those manifolds from those eight geometries, and to get answers to  those questions.  
 Let's see what happen when we restrict to each  individual geometry. 
 
 We first have a look of Question 1. Four cases are essentially known:  Each manifolds supports either $\Nil-$ or 
$ \wt{\PSL(2,\mathbb{R})}-$ geometry is non-achiral, so the density is 0.
There are only two orientable manifolds  supporting  $S^2\times E^1$-geometry, both are achiral, so the density is 1.
There are only seven orientable manifolds supporting the $ E^3$-geometry, four of them are achiral and three are not (\cite[Theorem 1.7]{SWW} and Lemma \ref{-1}), so the density is $4/7$.
It is not clear for us how to  order the manifolds for the remaining geometries to get the density. 

Now we discuss Questions 2 and 3, which is the issue of the paper. 
 The situation is known for the first six classes:  3-manifolds supporting one of the first six  geometries form only one commensurable class (\cite{Sc}), and
commensurable classes are achiral for the manifolds supporting first four geometries, and non-achiral for manifolds supporting 
next two geometries.
Moreover there is no non-orientable 3-manifolds supporting $S^3$-geometric, and there are  non-orientable 3-manifolds supporting the next three geometries. 
On the other hand  3-manifolds supporting one of the last two  geometries, Sol or  $\BH^3$, form infinitely many commensurable classes, and 
the achirality of those commensurable classes  are  not addressed.

We study the achirality of  commensurable classes of Sol 3-manifolds in this paper.

Each commensurable class $\CM$ of Sol 3-manifold contains an orientable torus bundle $M=M_\phi$, eigenvalues of whose monodromy map $\phi$ generates over $\BQ$ a real quadratic field. The   discriminant of the real quadratic field is the complete topological invariant of $\CM$, denoted by $D_{\CM}$ and called the discriminant of $\CM$ 

Our main result  is:

\begin{thm} \label{density} (1) Among all commensurable classes of Sol 3-manifolds, 
there are infinitely many achiral classes; however ordered by discriminants, the  density of achiral commensurable classes:

$$\lim_{X\ra \infty} \frac{\#\{\text{achiral}\ \CM\ \ |D_{\CM}<X\}}{\#\{\CM \ |\ D_{\CM}<X \ {}\}}=0,$$

(2) Among all achiral commensurable classes of Sol 3-manifolds ordered by discriminants, the density of classes containing non-orientable elements:
   $$\lim_{X\ra \infty} \frac {\#\{ \CM \ {}\text{contains an non-orientable element \,}|\, D_\CM <X\ \}}{\#\{\text{achiral}\ \CM\ \ |D_{\CM}<X\}}=1-\rho,$$
    where
$$\rho:=\prod_{j=1}^\infty  \left(1+2^{-j}\right)^{-1} = 0.41942\cdots.$$
\end{thm}

The project is started by wondering if each Sol 3-manifold has an achiral finite cover. 
Theorem \ref{density} (1) is unexpected to us which claims that to have achiral finite covers are rare among Sol 3-manifolds.
Theorem \ref{density} (2)  is a big surprise  that the Stevenhagen's conjecture, which was solved  by P. Koymans and C. Pagano \cite{KP} in 2022, has its avatar in topology.
  
For each orientable Sol torus bundle $M=M_\phi$, we introduce the discriminant $D_M$ of $M$ to be the discriminant of the fixed points of its monodromy $\phi$.  The integer $D_M$ is a topological invariant of $M$  so that $\BQ(\sqrt{D_M})$ is the real quadratic field of the 
commensurable class containing $M$.  
Another result, Theorem \ref{Gauss}, in this paper is that for any given $D\equiv 0, 1 \mod 4$ non-square positive integer, the following conditions are equivalent
 \begin{itemize}
 	\item there exists an achiral Sol torus bundle $M$ with discriminant $D$;
 	\item no prime factor of $D$ is $\equiv 3\mod 4$ and $16\nmid D$.  
 \end{itemize}
Consequences of Theorem \ref{Gauss} including:  among Sol 3-manifolds,  a commensurable class $\CM$ is achiral if and only if $D_{\CM}$ contains no prime $\equiv 3\mod 4$ (Theorem \ref{Gauss2}); each achiral commensurable class contains non-achiral manifolds (Corollary \ref{Gauss3});  
and $\CM$ contains an non-orientable element if and only if it contains an achiral torus semi-bundle (Corollary  \ref{Pell2}). 
 
In his study on cusp cross-sections of Hilbert modular varieties, Hirzebruch made a conjecture 
relates the signature defects of an orientable Sol torus bundle to the special value of its corresponding Shimizu's L-function. 
In this paper, we also show that, Theorem \ref{Shimizu} (1), for an oriented Sol torus bundle $M$,  $M$ is achiral if and only if  the Shimizu's L-series of $M$ is identically zero.
If monodormies of  two oriented Sol torus bundles  are negative to each other, they are not homeomorphic but their Shimuz L-series are the same.  We have that, Theorem \ref{Shimizu} (2), non-achiral oriented Sol tours bundles  are determined by their Shimizu L-functions up to such a relation. 

Theorem \ref{density} follows from the positive answer of Stevenhagen Conjecture, Theorem \ref{Gauss2}, and Corollary \ref{Pell2}
which claims that a commensurable class contains a non-orientable element if and only if the corresponding negative Pell equation
has a solution. Both the proofs of Theorem \ref{Gauss} and  \ref{Shimizu} are based on Theorem \ref{main1} which claims that 
a Sol 3-manifold $M_\phi$ is achiral if and only if $[-Q_\phi]=\pm [Q_\phi]$ in class group
$C(D_M)$
for its corresponding quadratic form $Q_\phi$.   The proof of Theorem \ref{Gauss} also highly relies  on Gauss' genus theory.


\section{Commensurable Classes of Sol 3-manifolds and Their Discriminants}

\s{\bf Classification of Sol 3-manifolds}.  Recall the Sol geometry is the Lie group $\BR^2\rtimes \BR^1$ with structure given by the representation $\BR^1\ra \GL_2(\BR), z\mapsto \matrixx{e^z}{}{}{e^{-z}}$, together the invariant Riemann metric $e^{2z} dx^2+e^{-2z} dy^2+dz^2$. A connected closed manifold is called Sol 3-manifold if it is modelled on $\Sol$, i.e.  is of form $\Sol/\Gamma$ for some discrete subgroup $\Gamma$ of $\Isom (\Sol)$.  There are two types Sol 3-manifolds:
\begin{itemize}
	\item  let $T=\BR/2\pi \BZ\times\BR/2\pi \BZ$ be a torus, $I=[0, 1]$, and $\phi\in\GL_2(\BZ)$ Anosov matrix, i.e. its eigenvalues are real but not $\{\pm 1\}$. Define
		\[M_\phi=T\times I/((x,y)\phi, 0)\sim (x,y,1).\]
		Thus $M_\phi$ is a tours bundle over a circle.
\item let  
		\[N=T\times I/(x,y,z)\sim(x+\pi,-y, 1-z)\]
		 a twisted $I$-bundle of the Klein bottle $K=T/(x, y)\sim (x+\pi, -y)$.  For each $\psi=\matrixx{a}{b}{c}{d}\in\SL_2(\BZ)$ with $abcd\neq 0$, define the semi-torus bundle by gluing two $N$'s along their  boundary $T\times \{1\}$ via $\psi$:
		\[N_\psi=N\cup N\big/ (x,y,1)\sim ((x,y)\psi, 1),\]
\end{itemize}

The following classification result of Sol 3-manifolds is basically known except the non-orientable case needs some explanations.


\begin{prop}\label{Class1}
 Any Sol 3-manifold is homeomorphic to one of the following  
\begin{enumerate}
		\item Torus bundles $M_\phi$:   two $M_\phi$ and  $M_{\phi'}$ are homeomorphic if and only if $\phi'$ is $\GL_2(\BZ)$-conjugate to $\phi^{\pm 1}$. Moreover, $M_\phi$ is orientable if and only if $\phi\in\SL_2(\BZ)$.
		
\item Torus semi-bundles $N_\psi$: two $N_\psi$ and $N_{\psi'}$are homeomorphic if and only if  $$\psi'=\matrixx{\pm 1}{0}{0}{\pm 1}\psi^{\pm 1}\matrixx{\pm 1}{0}{0}{\pm 1}.$$Moreover,  $N_\psi$ is always orientable. 
\end{enumerate}
\end{prop}

\begin{proof} Orientable Sol 3-manifolds are classified by (\cite{Ha}, Theorems 2.6 and 2.8, see also \cite{SWW} Theorems 2.3 and 2.4). Now let $M$ be a non-orientable Sol 3-manifold. Then the orientable double cover of $M$ must be either a torus bundle or torus semi-bundle. In the first case, $M$ must also be a torus bundle by the classification of free involutions on torus bundles (see \cite{Sak}). In fact, it is exactly the Case (II) in \cite{Sak} p 168. In the second case, $M$ is given by a free involution on torus semi-bundle which must be orientable by [Theorem 2] in \cite{BGV}.
\end{proof}

Fix orientations on $T$ and $I$, each orientable torus bundle $M_\phi, \phi\in \SL_2(\BZ)$, inherits an orientation, and also denote by $M_\phi$ the oriented torus bundle. 
\begin{lem}\label{inverse} For each oriented $M_\phi$, $M_{\phi^{-1}}=-M_\phi.$  where  $-M$ is the $M$ with opposite orientation.
\end{lem}

\begin{proof}
By definition we have  $$M_{\phi ^{-1}}=\frac {T\times [0 ,1]} {(x, 1)\sim (\phi^{-1}(x), 0)}.$$

Note identify the top to the bottom via $\phi^{-1}$ is the same as identify the bottom to top via $\phi$, and later is homeomorphic
to $M_\phi$ via an orientation reversing homeomorphism which is obtained by changing the orientation of $[0, 1]$.
We proved the lemma.
\end{proof}

 \begin{lem}\label{oriented}  Two oriented tours bundles $M_\phi$ and $M_{\phi'}$ are homeomorphic if and only if $\phi'$ is  $\SL_2(\BZ)$-conjugate to $\phi$ or $w\phi^{-1} w$, where $w=\matrixx{}{1}{1}{}$. 
 \end{lem}
 \begin{proof} By Proposition \ref{Class1} (1), $M_{\phi'}$ is homeomorphic to $M_\phi$ if and only if $\phi'=A\phi ^{\pm}A^{-1}, A\in\GL_2(\BZ)$.
 If $A\notin SL_2(\BZ)$, we write $A=Bw$ with $B\in SL_2(\BZ)$. So  $M_{\phi'}$ is homeomorphic to $M_\phi$ if and only if
  $\phi'$ is $\SL_2(\BZ)$ conjugate to $\phi^{\pm 1}$ or $w\phi^{\pm 1} w$. 
The conclusion follows by Lemma \ref{inverse} and that $\det w=-1$, that is $w$ is orientation reversing on $T$. 
 \end{proof}
 
 \s{\bf Commensurable classes of Sol 3-manifolds}. 
We call two manifolds  $M_1$ and $M_2$ commensurable, if they  have  a common finite cover.
It is easy to see that commensurable relation is an equivalent relation. We will use $\CM$ to denote the commensurable class of 
$M$. 

 Suppose $K=\BQ(\sqrt d)$ is a real quadratic field, where $d>0$ is square free.
Define the fundamental discriminant $D_K$ of $K$ by $D_K=d$ if 
$d\equiv 1 \, \text{mod} \, 4$ and $D_K=4d$ otherwise \cite[p.92]{EW}.

Each commensurable class $\CM$ of Sol 3-manifold contains an orientable torus bundle, eigenvalues of whose monodromy map generates over $\BQ$ a real quadratic field $K$. The   fundamental discriminant $D_K$, denoted by $D_\CM$ and called the discriminant of $\CM$.  


\begin{prop}\label{Class2} There is a one-to-one bijection between
\begin{itemize}
	\item commensurable classes $\CM$ of Sol 3-manifolds;
	\item real quadratic fields $\BQ(\sqrt {D_\CM})$ (or equivalently, fundamental positive discriminants $D_\CM$),
\end{itemize}
\end{prop}

For each torus bundle $M_\phi: T\to M_\phi \to S^1$, we have a short fiber exact sequence
	
	$$1\to \pi_1(T)\to \pi_1(M_\phi)\to \pi_1(S^1)\to 1. \qquad (2.1) $$
	
	 For each covering map $p: T\to T$, we can homotopy $p$ to be non-degenerated linear map, that is   $p \in M_2(\BZ)$ and $p$ is of rank $2$.  For each fiber preserving covering map $f: M_\psi\to M_\phi$ between torus bundles,  there is an induced commutative 
	diagram between of these exact sequences.
	
	$$\xymatrix{
1\ar[r] &    \pi_1(T) \ar[r]^{i'_*} \ar[d]^{f|_*} &   \pi_1(M_\psi)  \ar[r]^{q'_*} \ar[d]^{f_*} & \ar[d]^{\bar f_*}  \pi_1(S^1) \ar[r] & 1\\
1\ar[r] &      \pi_1(T) \ar[r]^{i_*}                   & \pi_1(M_\phi)     \ar[r]^{q_*}                  &  \pi_1(S^1)  \ar[r]& 1   }   \qquad (2.2)                                                                $$
	
	Call the covering $f$ is vertical, if the left side vertical map $f|_*$ is an isomorphism, that is to say the covering is from the circle direction;
	and the covering is horizontal, if the right side vertical map $\bar f_*$ is an isomorphism, that is to say the covering is from the torus direction.
	
	Call a closed orientable 3-manifold $M$  Haken, if each embedded 2-sphere in $M$ bounds a 3-ball in $M$, and there 
exists a closed orientable embedded surface of genus $\ge 1$ in $ M$ which induce an injection $\pi_1$. 
Each Sol 3-manifold is Haken.  

\begin{lem}\label{v} Suppose $M_\phi$ and $M_\psi$ are orientable torus bundle supporting Sol geometry. Then 

(1) Each Sol $M_\phi$ has unique torus bundle structure up to isotopy.

(2) Any covering $f: M_\psi \to M_\phi$ is fiber preserving up to isotopy.

\end{lem}

\begin{proof} 
(1) Suppose $T'$ is a torus fiber of a fiberation of $M_\phi$. According to (2.1), $T\to M_\phi$ induces an  injection
	on $\pi_1$, hence $T'$ is incompressible. Since $|\tr (\phi)|>2$, $\phi$ is not conjugate to $\matrixx{1}{0}{n}{1}$, 
	by \cite[Lemma 5.2]{Ha}, $T'$ is isotopy to $T$, the torus fiber of the fiberation determined by $\phi$. 
	
(2) If we fiber $M_\psi$ with lifted torus fiberation of $M_\phi$, then the covering is fiber preserving. Since $M_\psi$ also  supports Sol geometry, by (1) its original torus fiberation is isotopic to the lifted torus fiberation, and (2) follows.
\end{proof}

\begin{lem}\label{h-v} For each fiber preserving covering map $f: M_\psi\to M_\phi$ between torus bundles, $f=f_v\circ f_h$,
	where $f_h$ is a horizontal covering and $f_v$ is a vertical covering. Moreover \begin{itemize}
	\item[(1)] there is a horizontal covering $f: M_{\psi}\ra M_{\phi}$ if and only if $\psi$ is $\GL_2(\BQ)$-conjugate to $\phi^{\pm 1}$, or equivalently, $\tr \phi=\tr \psi$
	\item[(2)] there is a vertical covering $f: M_{\psi}\ra M_{\phi}$ of degree $n$ if and only if $\psi$ is $\GL_2(\BZ)$-conjugate to $\phi^{\pm n}$. 
\end{itemize}
It follows that any fiber preserving covering $f: M_{\psi}\ra M_{\phi}$ has that $\psi$ is $\GL_2(\BQ)$-conjugate to $\phi^n$ for some integer $n\neq 0$.  
\end{lem}

\begin{proof} Let $G=(q^*)^{-1}\bar f_*(\pi_1(S^1))$,  $G$. By (2.2), 
	 $G$ is subgroup of $\pi_1(M_\phi)$
	and $f_*(\pi_1(M_\psi))$ is a subgroup of $G$. Since $f: M_\psi \to M_\phi$ is a finite covering,  we have that  $G$ is a finite index subgroup of $\pi_1(M_\phi)$.
	By covering space theory,  there is a 3-manifold $M$ with $\pi_1(M)=G$ and the finite covering $f$ is docomposed into the 
	 finite covering $f_h: M_\psi \to M$ and $f_v: M \to M_\phi$. If we lift the torus bundle structure of $M_\phi$ to $M$, then both
	 $f_h$ and $f_v$ become a fiber preserving covering. One can verify that $f_h$ is  horizontal and $f_v$ is  vertical.
	 
We now show (1). Let $T$ be a  fiber torus of $M_\phi$. Since the fiber preserving covering $f$ is horizontal, $\tilde T=f^{-1}(T)$ is a fiber torus of $M_\psi$. Cutting $M_\phi$ along $T$ and $M_\psi$ along $\tilde T$,  we get an induced  a fiber preserving covering map
$$\bar f : \tilde T \times [0,1]\lra  T \times [0,1],$$
where $f(x, t)=(p(x), \epsilon t)$,  where $p: \tilde T\to T$  is a covering map, and $\epsilon=\pm 1$ which depends on the map preserving or reversing the orientation of $[0,1]$. Then from the constructions of $M_\psi$  and  $M_\phi$, we have the following commutative diagram
$$\xymatrix{
\tilde T \times \{1\} \ar[r]^{p} \ar[d]^{\psi} &   T \times \{1\}   \ar[d]^{\phi^{\epsilon}} \\
\tilde T \times \{0\} \ar[r]^{p}                   &  T \times \{0\}                      }.                                                                 $$
That is to say for any $(x, 1)\in \tilde T \times \{1\}$,  we have 
$\phi^{\epsilon} \circ p (x, 1)=p\circ \psi (x, 1),$
where we define $p(x, 1)=(p(x), 1)$. That is 
$(\phi^{\epsilon}\circ p (x), 0)=(p\circ \psi (x), 0).$ So we have $\phi^\epsilon\circ p=p\circ \psi$, or 
$$A\circ \psi =\phi^{\epsilon}\circ A, \qquad (2.3)$$ 
where $A$ is the rank 2 element  in $M_2(\BZ)$ induced by $p$.  Another direction is a direct construction.

It is clear that $\psi =A^{-1}\circ\phi^{\pm 1}\circ A$ implies that  $\tr(\phi)=\tr(\psi)$. On the other hand, if $\tr(\phi)=\tr(\psi)$, then $\phi$ and $\psi$ are conjugate in $\GL_2(\BR)$ and thus in $\GL_2(\BQ)$ and in  $\GL_2(\BZ)$ (the equation $p\phi=\psi p$ in $p$ is linear). In fact, let $\phi=\matrixx{a}{b}{c}{d}$ and $\psi=\matrixx{a+u}{w}{v}{d-u}$, then $cv\neq 0$ and one may take $p=\matrixx{1}{u/c}{0}{v/c}$.

We now show (2). First $M_{\phi^n}$ has a vertical covering structure of $M_\phi$ of degree $n$. On the other hand, $H_1(M_\phi, \BZ)=\BZ \oplus \frac{H_1(T, \BZ)}{\Im (\phi_*-\mathrm{Id})}$ and a vertical covering $M_\psi\ra M_\phi$ of degree $n$ corresponds to the index $n$ subgroup of $H_1(M_\phi, \BZ)$ given by the $\BZ$-component. By Lemma \ref{v} (1),  the uniqueness of fiberation, $\psi$ is conjugate to $\phi^{\pm n}$. 	
\end{proof}

Now we recall some facts about  real quadratic field $K=\BQ(\sqrt d)$ \cite[Chapters 3 and  4]{EW}which will be used in the proof of Proposition \ref{Class2} and thereafter. For each $x=a+b \sqrt d \in F$, define its norm $N(x)=a^2-b^2d$. Call $x$ absolutely positive if both
$x$ and $N(x)$ are positive.

Let $O_K$ be the ring consists of all algebraic integers of $K$, or equivalently, the $x\in K$ whose minimal polynomial
hase form $x^2+bx+c$ for $b, c \in \BZ$. 

Call a subring $\CO\subset O_K$ of rank 2 an order of $K$. It is known that $\CO=\BZ\oplus c^2O_K$ for $c\in \BZ$.
Define the discriminant of $\CO$ to be $c^2D_F$. Note $O_K$ itself is the maximum order of $K$.
Let $\CO^\times$ be group of all units of $\CO$. Then  Dirichlet Unit Theorem implies $\CO^\times=\BZ\oplus \BZ_2$, and call a positive generator of $\BZ$ a fundamental unit of $\CO$. 
 
\begin{proof}[Proof of Proposition \ref{Class2}]  Suppose $M_{\phi_1}, M_{\phi_2}$ are commensurable. Then there are two coverings
  $M_{\phi}\ra M_{\phi_1}, M_{\phi_2}$ between Sol torus bundles. By Lemma \ref{v}, we may assume that they are fiber preserving
  coverings. By Lemma \ref{h-v}, we have 
$$\phi=A_i \phi_i^{n_i} A_i^{-1}, \qquad i=1, 2.$$
Here $A_i\in \GL_2(\BQ)$ and $n_i\neq 0$ integers. It follows that $$\phi_2^{n_2}=(A_2^{-1} A_1) \phi_1^{n_1}(A_1^{-1} A_2).$$ Therefore, $\phi_1, \phi_2$ give rise to the same real quadratic fields. 

On the other hand, suppose $\phi_1, \phi_2$ give rise to the same real quadratic field $K$, and let $\lambda_1, \lambda_2$ be their eigenvalues, respectively. 
As the roots of characteristic polynomial of element in $SL_2(\BZ)$, both $\lambda_1, \lambda_2$ are  units  in $O_F$.
Let $\epsilon$ be a fundamental unit of $O_K$.
Since $O_K^\times =\BZ\oplus \BZ_2$,  we have 

$$\lambda_1=\pm \epsilon^{n_1}, \quad \lambda_2=\pm \epsilon^{n_2}$$
for some integers $n_1, n_2\neq 0$.
This implies $\lambda_1^{2n_2}=\lambda_2^{2n_1}$, and therefore $\tr (\phi_1^{2n_2})=\tr (\phi_2^{2n_1})$. By Lemma \ref{h-v} (1), there exists a covering $M_{\phi_1^{2n_2}}\ra M_{\phi_2^{2n_1}}$. Since $M_{\phi_2^{2n_1}}$ covers $M_{\phi_2}$, $M_{\phi_1^{2n_2}}$
covers both $M_{\phi_1}$ and $M_{\phi_2}$.
It follows that $M_{\phi_1}$ and $M_{\phi_2}$ are commensurable. 
\end{proof}



\begin{lem}\label{-1} For a Haken manifold $M$,  $M$ is achiral if and only if $-1$ is a mapping degree of $M$.
\end{lem}		
	
\begin{proof} Suppose $f: M\to M$ is a map of degree $-1$.
Then  $f$ induces a surjection
on $f_*: \pi_1(M)\to \pi_1(M)$ \cite[page 175]{He1}.  
Since $M$ is Haken, $\pi_1(M)$ is Hopfian,  which first implies that $f_*$ is an isomorphism  \cite[page 177]{He1}, \cite{He2}, then  implies that
$f$ is homotopic to a homeomorphism \cite[Chap. 13]{He1}, which must be  orientation reversing.  Another direction is obvious.
\end{proof}

\begin{lem}\label{semi-torus bundle} 
Each torus semi-bundle $N_\psi$ has a  unique torus bundle double cover $M_\phi$. The covering is characteristic and has a realization as $M_\phi\ra N_\psi$ with $\psi=\matrixx{a}{b}{c}{d}$ and $\phi=\matrixx{2bc+1}{2ab}{2cd}{2bc+1}$. 
\end{lem} 

\begin{proof}
 It is known that each $N_\psi$ is doubly covered by a torus bundle $M_\phi$ (\cite{Ha}, Theorems 2.6 and 2.8).  We just verify that $M_\phi$ is unique 
and characteristic.

For each torus semi-bundle $N_\psi$, we have a short fiber exact sequence
		$$1\to \pi_1(T)\to \pi_1(N_\psi)\to \BZ_2*\BZ_2\to 1. \qquad (2.4)$$	
The unique torus bundle double cover $f: M_\phi\to N_\psi$ is given by the unique index 2  subgroup $\BZ\subset \BZ_2*\BZ_2$,
which is also characteristic, indeed $\BZ_2*\BZ_2$ has only three index 2 subgroups, two of them isomorphic to $\BZ_2*\BZ_2$, one is isomorphic to 
$\BZ$.
  There is an induced commutative 
	diagram between of these exact sequences.
	
$$\xymatrix{
1\ar[r] &    \pi_1(T) \ar[r]^{i'_*} \ar[d]^{f|_*} &   \pi_1(M_\phi)  \ar[r]^{q'_*} \ar[d]^{f_*} & \ar[d]^{\bar f_*}  \BZ\ar[r] & 1\\
1\ar[r] &      \pi_1(T) \ar[r]^{i_*}                   & \pi_1(N_\psi)     \ar[r]^{q_*}                  &  \BZ_2*\BZ_2  \ar[r]& 1   }   \qquad (2.5)                                                                $$
where $f|_*$ is an isomorphism. Then we can identify $\pi_1(T)\subset \pi_1(M_\phi)$ and $\pi_1(T)\subset \pi_1(N_\psi)$.
Note any homeomorphism on $\tau$ on $N_\psi$ keeps $T=\partial N$ invariant up to isotopy and $\pi_1(T)$ is a normal subgroup
of $\pi_1(N_\psi)$.
It follows that $\pi_1(T)\subset \pi_1(N_\psi)$ is characteristic. Then for any automorphism $\eta$ on $\pi_1(N_\psi)$,
 $\eta$  induces an automorphism $\bar \eta$ on $\BZ_2*\BZ_2$. Since $\bar \eta(\BZ)=\BZ$, we have  
 $q^{-1}(\bar \eta(\BZ))=q^{-1}(\BZ)$. Since $q^{-1}(\bar \eta(\BZ))=\eta (q^{-1}(\BZ))$, we have $\eta (q^{-1}(\BZ))=q^{-1}(\BZ)$.
 That is $\pi_1(M_\phi)=q^{-1}(\BZ)$ is characteristic.
\end{proof}

\begin{lem}\label{virtual}
$\CM$ is achiral if and only if each element in $\CM$ is virtually achiral.
\end{lem}

\begin{proof} Suppose  $\CM$ is achiral, we may assume that $M$ is achiral, that is there is an orientation reversing homeomorphism
$\tau$ on $M$. For any $M'\in \CM$, let $\tilde M$ be a common finite cover
of $M$ and $M'$. Suppose the covering degree $\tilde M\to M$ is $d$, then the index $[\pi_1(M): \pi_1(\tilde M )]=d$. Since $G=\pi_1(M)$ is finitely generated,  by group theory $G$ has only finitely many subgroup of indes $d$.
Let $$H= \bigcap_{\{G_i, [G: G_i]=d\}} G_i.$$
Then $H$ is a finite index charecteristic  subgroup of $G$. Let $M_H\to M$ be the covering corresponding to $H$.
Since $\tau_*(H)=H$, $\tau$ lifts to a map $\tau_H$ on $M_H$, which must be an orientation reversing homeomorphism, that is $M_H$ is achiral.
Clearly $M_H$ covers $\tilde M$, so covers $M_2$. That is $M_2$ is virtually achiral. We prove one direction.

Another direction is from the definition.
\end{proof}

 \section{Achirality, Gauss'  Theory of Class Groups and  Genus}
  
A closed oriented manifold is called achiral if it admits an orientation reversing homeomorphism. 

\begin{lem}\label{pmA} Let $M=M_\phi$ be an oriented Sol torus bundle. Then $M_\phi$ is achiral if and only if there exists $A\in \GL_2(\BZ)$ such that one of the following holds. \begin{enumerate}\item[(i)] $\phi A=A\phi$,  $\det A=-1$. \item[(ii)] $\phi A=A\phi^{-1}$, $\det A=1$.\end{enumerate}  \end{lem}

\begin{proof} Let $f$ be a self homeomorphism on $M_\phi$. By Lemma \ref{v}, we assume that $f$ is fiber preserving, so
$f$ is a fiber preserving horizontal covering,  and  (2.3) in  the proof of Lemma \ref{h-v} become  
$$A\circ \phi =\phi^{\epsilon}\circ A, $$ where $A\in \GL_2(\BZ)$ since the restriction of $f$ on the fiber $T$ is a homeomorphism.

It is clear that $\epsilon \cdot \det A=-1$ if and only if $f$ is orientation reversing. 
\end{proof}

We now recall  Gauss'  of class group theory  \cite{Gauss}, \cite[Chapter 5]{Fl}, \cite[Chapters 3, 4]{EW}.

Recall $\SL_2(\BZ)$ acts on the complex plane  by 
$$\phi z=\frac{az+b}{cz+d}, \qquad \forall \phi=\matrixx{a}{b}{c}{d}\in \SL_2(\BZ),\ z\in \BC, $$ which induces an action of $\SL_2(\BZ)$ on $\BR$.  There are  two conjugate real quadratic numbers $z_\phi, \bar z_\phi$ fixed by $\phi$ when  $|\tr \phi|>2$. 

Let $M=M_\phi, \phi=\matrixx{a}{b}{c}{d}\in \SL_2(\BZ),$  $|\tr \phi|>2$. 
Let $$u=\sign \tr(\phi)\gcd (b, c, (d-a)), \,\text{and}\,   \,\alpha=\frac{c}{u}, \,\beta= \frac{d-a}{u}, \,\gamma=\frac{-b}{u}.$$
We define the discriminant of $\phi$ and of $M$ to be
$$D_M=D_\phi:= \frac{(d+a)^2-4}{u^2}=\beta^2-4\alpha\gamma>0,\qquad (3.1)$$

It is easy to see that $D=D_\phi \equiv 0, 1\mod 4$ is a non-square positive integer. 

\begin{lem} \label{realization of D}Let $D\equiv 0, 1\mod 4$ be a non-square positive  integer. Then there exist $M_\phi$ such that $D_\phi=D$.
\end{lem}
\begin{proof} It is known that the Pell equation $x^2-dy^2=1$, for each integer $d>0$, has (infinitely many) non-trivial integer solutions.
 Consider the following Pell equations:
$$x^2-\frac{D}{4}y^2=1 \quad (4|D), \qquad x^2-Dy^2=1\quad  (4\nmid D).$$
 Let 

$$\phi=\matrixx{w}{\frac{D}{4}v}{v}{w}\quad  (4|D), \qquad \matrixx{w-v}{\frac{D-1}{2}v}{2v}{w+v}\quad  (4\nmid D),$$
where $(w, v)$ is a non-trivial integer solution of the corresponding Pell equation.
Clearly $\phi\in  \SL_2(\BZ)$. One can also chech by the definition that $D_\phi=D$.
\end{proof}

Let $D\equiv 0, 1\mod 4$ be a non-square integer. Define  
primitiive binary quadratic form $Q_\phi$ with discrimiment $D$  given by
$$Q_\phi(x,y):=\frac{cx^2+(d-a)xy-by^2}{u}=\alpha x^2+\beta xy + \gamma y^2.\qquad(3.2)$$
Note $\alpha x^2+\beta x +\gamma$ is a minimal polynomial of the fixed point of $\phi$, which is unique up to a sign,
and is determined by the fixed point of 
$\phi$. 

\begin{lem}\label{Q and -Q}
For each $\phi\in SL_2(\BZ)$, (i) $Q_{-\phi}=Q_\phi$, (ii) $Q_{\phi^{-1}}=-Q_\phi$,  and (iii) $Q_{\phi^{n}}=Q_\phi$ for $n>0$, which is  equivalent to  $ Q_{\pm \phi^n}=\text{sign}(n)Q_\phi$ for $n>0$.
\end{lem}
\begin{proof} 
It follows from a direct calculation. A quick way to prove $Q_{\phi^{n}}=Q_\phi$ for $n>0$
is apply the sentence before the lemma and the fact $\phi$ and $\phi^n, n\ne 0$, has the same fixed points.
\end{proof}  

For each $A\in \GL_2(\BZ)$ one can verify that $$Q_{A\phi A^{-1}}(x, y)=Q_\phi ((x,y)(A^t)^{-1}). \qquad(3.3)$$
Call two primitive binary quadratic forms $Q_1$ and $Q_2$ with discriminant $D$ are $SL_2(\BZ)$-equivalent, if there exists $A\in SL_2(\BZ)$,  such that
$$Q_1(x,y)=Q_2 ((x,y)A).$$ 
 

Those $SL_2(\BZ)$-equivalent classes  of primitive binary quadratic forms of discriminant $D$  is the class group $C(D)$. For a quadratic form $Q$, denote by $[Q]$ its class in  $C(D)$. Let $Q_{0, D}$ denote the principal form of discriminant $D$, i.e.
$$x^2-\frac{D}{4}y^2, \ (\ \text{if}\ 4|D\ ); \qquad x^2+xy+\frac{1-D}{4}y^2\ (\ \text{if}\ 4\nmid D\ ).$$
Then $[Q_{0, D}]$ is the zero element in $C(D)$. For any form $Q=(\alpha, \beta, \gamma):=\alpha x^2+\beta xy+\gamma y^2$,  let 
$-Q=(-\alpha, -\beta,-\gamma)$, for $[Q]\in C(D)$, let $-[Q]$ denote its negative in $C(D)$. 


 
 


 \begin{lem}\label{rel-in C(D)} For  forms of discirminant $D$,
  
 (i) $[Q]=[Q_{0, D}]$ is equivalent to $Q$ represents $1$, $[Q]=[-Q_{0, D}]$ is equivalent to $Q$ represents $-1$;
 
 (ii) $[-Q_{0, D}]=[Q_{0, D}]$ is equivalent to $x^2-Dy^2=-4$ has an integral solution;
 
 (iii) $[-Q]=-[Q]+[-Q_{0, D}]$; 
 
 (iv) $[-Q]=\pm [Q]$ is equivalent to either $[Q_{0, D}]=[-Q_{0, D}]$ or $2[Q]=[-Q_{0, D}]$;
 
(v)  $[Q_1]=-[Q_2]$  if and only if $Q_1(x, y)=Q_2((x, y)B)$, where $B\in \GL_2(\BZ)$ and $\det B=-1$.
\end{lem}


Lemma \ref{rel-in C(D)} will be used repeatedly below.
Its proof is given in the appendix. 




\begin{thm}\label{main1}
	Let $M=M_\phi$ be an oriented Sol torus bundle with $D=D_\phi$ defined in (3.1) and  $Q_\phi$ defined in (3.2). The following are equivalent.
	\begin{enumerate}
		\item the Sol 3-manifold $M$ is achiral.
		\item either $[Q_{0, D}]=[-Q_{0, D}]$ or $2[Q_\phi]=[-Q_{0, D}]$ (or quivalently $[-Q_\phi]=\pm [Q_\phi]$).
\end{enumerate}
\end{thm}
We first to prove the following 
\begin{lem}\label{**}
If $[Q_{\phi_1}]=[Q_{\phi_2}]$ for two anosov maps $\phi_1, \phi_2\in SL_2(\BZ)$, then there exists $\alpha\in SL_2(\BZ)$ and integers 
$m, n>0$ such that
$\phi_1$ is $SL_2(\BZ)$-conjugates $\pm \alpha^{m}$ and $\phi_2$ is $SL_2(\BZ)$-conjugates $\pm \alpha^{n}$. 
\end{lem}

\begin{proof} We first claim:
If $Q_{\phi_1}=Q_{\phi_2}$ for two anosov maps $\phi_1, \phi_2\in SL_2(\BZ)$, then there exists $\alpha\in SL_2(\BZ)$  and integers 
$m, n>0$ such that $\phi^t_1=\pm \alpha^{m}, \, \phi^t_2=\pm \alpha^{n}$.

For each primitive binary quadratic form $Q$, define 
$$Aut^+(Q)=\{\beta\in SL_2(\BZ) | Q(x,y)=Q((x,y)\beta)\}$$

Suppose  first $Q_{\phi_1}=Q_{\phi_2}=Q$, then it is a direct verification that $\phi_1^t, \phi_2^t \in Aut^+(Q)$.
By  \cite[Corollary 5.9]{Fl}, which claims that
$Aut^+(Q)=\BZ\times\{\pm 1 \}=\{\pm \alpha^n\},$
, we have 
$$\phi^t_1=\pm \alpha^{m}, \, \phi^t_2=\pm \alpha^{n}$$  for integers $m, n\ne 0$. 
Now $m$ and $n$ must have the same sign, otherwise may assume that $m>0$ and $n<0$,   we have 
$$Q=Q_{\phi_1}=Q_{\pm (\alpha^{t})^m}=Q_{(\alpha^{t})^m}=Q_{\alpha^{t}}=-Q_{ (\alpha^t)^{-1}}=-Q_{ (\alpha^t)^{(-1)(-n)}}=-Q_{ (\alpha^t)^{n}}=-Q_{ \phi_2}=-Q,$$
 where the 3rd, 4th, 5th, and 6th equalities follow from  Lemma \ref{Q and -Q} (i), (iii) (ii) and (iii)  respectively.
and we reach  a contradiction. We may assume that $m, n>0$, otherwise replace $\phi$ by $\phi^{-1}$. So the claim is verified.

Now suppose $[Q_{\phi_1}]=[Q_{\phi_2}]$. From the definition, we have  $Q_{\phi_1}(x, y)=Q_{\phi_2}((x,y) \beta^t)$ for 
some $\beta\in SL_2(\BZ)$. By (3.3) we have $Q_{\phi_1}((x,y) \beta^t)= Q_{\beta^{-1}\phi_2\beta}(x.y)$. So we
have $Q_{\phi_1}=Q_{\beta^{-1}\phi_2\beta}$. Now apply the claim to $Q_{\phi_1}=Q_{\beta^{-1}\phi_2\beta}$, we proved  Lemma \ref{**}.
\end{proof} 

\begin{proof}[Proof of Theorem \ref{main1}] Assume that $M_\phi$ is achiral. With notations in Lemma \ref{pmA},  let 
$$A=\matrixx{x}{z}{y}{w}, \phi=\matrixx{a}{b}{c}{d}.$$  The conditions (i) and (ii) is equivalent to that
$$A=\begin{cases} \matrixx{x}{\frac{by}{c}}{y}{\frac{(d-a)y+cx}{c}}, \quad \det  A=-1, \qquad &\text{for (i)}\\
	\matrixx{x}{\frac{(d-a)x-by}{c}}{y}{-x}, \quad  \det A=1 \qquad &\text{for  (ii)}.
\end{cases}
$$
It is easy to see that the condition (i) is equivalent to that $Q_{0, D}$ represents $-1$, therefore equivalent to
$[Q_{0, D}]=[-Q_{0, D}]$ by Lemma \ref{rel-in C(D)} (i).

Below we prove that  Lemma \ref{pmA} (ii) is equivalent to that $[Q_\phi]=[-Q_\phi]$, that is, 
$[Q_\phi]=[-Q_\phi]$ if and only if $\phi=\beta \phi^{-1} \beta ^{-1}$ for some $\beta\in SL_2(\BZ)$.

The "if" part follows from the definition and formula (3.3).

Now suppose  $[Q_\phi]=[-Q_\phi]$. Then  $[Q_\phi]=[Q_{\phi^{-1}}]$ by Lemma \ref{Q and -Q}.
Then by Lemma \ref{**}, there exist $\alpha\in SL_2(\BZ)$ and integers $m, n>0$ such that 
$\phi$ is $SL_2(\BZ)$-conjugates $\pm \alpha^{m }$ and $\phi^{-1}$ is $SL_2(\BZ)$-conjugates $\pm \alpha^{n}$.
Since $\phi$ and $\phi^{-1}$ have the same eigenvaluse, it concludes that $m=n$ and the sign before $\alpha$ are the same.
So $\phi=\beta \phi^{-1} \beta ^{-1}$ for some $\beta\in SL_2(\BZ)$.
\end{proof}

We now recall Gauss' genus theory (see \cite{Gauss}, \cite[Chapter 5]{Fl} and \cite{BS}).  Let $D\equiv 0, 1\mod 4$ be a non-square positive integer, 
the Kronecker symbol $\chi_D: (\BZ/D\BZ)^\times \ra \{\pm 1\}$ is a homomorphism  such that $\chi_D(p\mod D)=\left(\frac{D}{p}\right)$ for all primes $p\nmid 2D$, where $\left(\frac{D}{p}\right)$, the Legendre symbol,  is 1 if and only if $D$ is quadratic residue modulo $p$.  Let $H_D\subset \Ker \chi_D$ be the subgroup consisting of values represented by the principal form of discriminant $D$.  Then there is an exact sequence 
$$0\ra 2C(D)\lra C(D)\stackrel{\omega}{\lra} \Ker \chi_D/H_D\lra 0,\qquad (3.6)$$
where $\omega$ sends a class to the coset of $H_D$ in $\Ker \chi_D$ it represents. 


The group $H_D$ consists of elements $[a]\in (\BZ/D\BZ)^\times$ such that 

$$\left(\frac{a}{p}\right)=1\, \, \text{for all odd prime divisors $p$ of $D$} \qquad (3.7)$$
 and such that
$$a\equiv \begin{cases}
 1, 7\mod 8, \qquad &\text{if $D\equiv 8\mod 32$},\\
 1\mod 4, \qquad &\text{if $D\equiv 12, 16, 28\mod 32$}, \,\,\, \qquad (3.8)\\ 
 1, 3 \mod 8, &\text{if $D\equiv 24\mod 32$},\\
 1\mod 8, &\text{if $32|D$}. 
 \end{cases}$$
 


\begin{thm}\label{Gauss} Let $D\equiv 0, 1\mod 4$ be a positive non-square integer. Then the following are equivalent:
\begin{enumerate}

\item there exists an achiral  oriented Sol torus bundle $M$ of discriminant $D$,

\item  $16\nmid D$ and no prime factor of $D$ is $\equiv 3\mod 4$. 
\end{enumerate}	
\end{thm}
\begin{proof}  
We first verity that (1) is equivalent to the following 

(3)  there exists a $[Q]\in C(D)$ such that $2[Q]=[-Q_{0, D}]$ (or equivalently $2[Q]$ represents $-1$). 

The equivalence of (1) and (3) follows from Theorem \ref{main1}:
If there exists a $[Q]\in C(D)$ such that $2[Q]=[-Q_{0, D}]$, then the Sol 3-manifold $M_\phi$ with $Q_\phi=Q$ is achiral by Theorem \ref{main1}.
On the other hand if $M_\phi$ with $D_\phi=D$ is achiral, then either $[Q_{0, D}]=[-Q_{0, D}]$ or $2[Q_\phi]=[-Q_{0, D}]$ by Theorem \ref{main1}.
 Note if $[Q_{0, D}]=[-Q_{0, D}]$, choose $\phi_0$ such that $Q_{\phi_0}=Q_{0, D}$.
Then $2[Q_{\phi_0}]=2[Q_{0, D}]=[Q_{0, D}]=[-Q_{0, D}]$.

Below we will prove that (2) and (3) are equivalent  in two steps by using Gauss' genus theory.

{\bf Step 1}: To verify that (3)  $2[Q]=[-Q_{0, D}]$ ($2[Q]$ represents $-1$) is equivalent to $[-1]\in H_D$ by using the short exact sequence (3.6):

Suppose $-1$ is represented by $2[Q]$. Since $2[Q]\in 2C (D)$, we have $\omega(2[Q])=\bar 1\in \Ker \chi_D/H_D$, which implies
$[-1]\in H_D$.

Suppose $[-1]\in H_D$. We only need to prove $\omega([-Q_{0, D}])==\bar 1\in \Ker \chi_D/H_D$, therefore $[-Q_{0, D}]=2[Q]$
for some $[Q]\in  C(D)$. If not, we have $\omega([-Q_{0, D}])=gH_D\in \Ker \chi_D/H_D$, where $g\notin H_D$. Since $[-1]\in H_D$,
$[-1]\notin gH_D$, which implies $-Q_{0, D}$ does not represent any $a\equiv 
 -1\mod D$, contradicting that $-Q_{0, D}$ represents $-1$.
 
{\bf Step 2}: To verify $[-1]\in H$ and (2) are equivalent by using (3.7) and (3.8).

For any odd prime $p|D$, by Euler's Theorem, $(\frac {-1}p)=1$ if and only if $p$ is not $ \equiv 3\mod 4$.

Suppose $[-1]\in H_D$. 
By Euler's Theorem and (3.6) (a characterization of $H_D$), we have  that  no prime factor of $D$ is $\equiv 3\mod 4$.  
Since $-1$ is not $ \equiv 1\mod 8$, by the fourth line of (3.7) (another  characterization of $H_D$), we have $D$ is not a multiple of 32.
Since $-1$ is not $ \equiv 1\mod 4$, by the second line of (3.7), we have $D$ is not congruent to 16 mod 32.
Hence $D$ is not a multiple of 16.

Now suppose $16\nmid D$ and no prime factor of $D$ is $\equiv 3\mod 4$.  
Then no prime factor of $D$ is $\equiv 3\mod 4$ implies that for any odd prime $p|D$,  $\left(\frac{-1}{p}\right)=1$ \cite[Page 71]{Fl}. Moreover suppose
$$D=2^\alpha\prod_{i=1}^k (4n_i+1)^{\beta_i}$$
is  the prime decomposition of $D$. Then $D=2^\alpha (4n+1)$ for some $n\in \BZ$.
Since $16\nmid D$, we have $\alpha \le 3$.
We have $4\nmid D$ if $\alpha \le 1$, and  $D\equiv 4\mod 16$ if $\alpha =2$, and $D\equiv 8\mod 32$ if $\alpha =3$.
Overall we have that $D$ is not $0, 16, 12, 24, 28,   \mod 32$.
Then by $-1$ appears in (3.8) for other $D$. And by the  charaterization of $H_D$, $[-1]\in H_D$.
\end{proof}


\begin{thm}\label{Gauss2} A commensurable class $\CM$ of Sol 3-manifolds is achiral if and only if $D_{\CM}$ 
contains no prime $\equiv 3\mod 4$.
\end{thm}

\begin{proof} Suppose $\CM$ is achiral, then there is achiral Sol 3-manifold $M\in \CM$. If $M$ is a semi-bundle. Let $\tilde M$ be the torus bundle double cover of $M$. Then $\tilde M\in \mathcal{M}$. Since $M$ is achiral and $\tilde M\rightarrow M$ is a characteristic  cover (Lemma \ref{semi-torus bundle}),  $\tilde M$ is also achiral. 
 So way assume that  $M$ is a torus bundle.
Then $D_{M}$ contains no prime $\equiv 3\mod 4$ by Theorem \ref{Gauss}.  
Since $D_{M}=c^2D_{\CM}$ $(c\in \BZ)$, $D_{\CM}$ contains no prime $\equiv 3\mod 4$. 

From the definition of $D_{\CM}$ as fundamental discriminant   we know that $16\nmid D_{\CM}$.
If $D_{\CM}$ contains no prime $\equiv 3\mod 4$, then there is an achiral Sol 3-manifold $M\in \CM$ with $D_{M}=D_{\CM}$ by 
Theorem
\ref{Gauss}. So $\CM$ is achiral.
\end{proof}

\begin{cor}\label{Gauss3}  
Each commensurable class of Sol 3-manifolds contains non-achiral manifolds.
\end{cor}
  
\begin{proof}  For each commensurable class $\CM$ of Sol 3-manifolds and each integer $c\ne 0$, $c^2D_{\CM}$  is a non-square positive  integer $\equiv 0, 1\mod 4$. Then  there is a $M\in \CM$
such that $D_{M}=c^2D_{\CM}$ by Lemma \ref{realization of D}. If we choose $c$ to be either 3 or 4, then $M$ is non-achiral by Theorem \ref{Gauss}. 
\end{proof}

 \section{Density Results}
 \begin{prop}\label{Pell} Let $D\equiv 0, 1\mod 4$ be a positive non-sqaure integer. Then the following are equivalent:
 \begin{enumerate}
 	\item The negative Pell equation $x^2-Dy^2=-4$ has solutions;
 	\item There is a non-oriented Sol 3-manifold with characteristic double cover of discriminant $D$;
 	\item There is an achiral semi-tours bundle with characteristic double cover of discriminant $\lcm(D, 4)$.
 \end{enumerate}
 	\end{prop}

\begin{proof} Assume (1), or equivalently, by Lemma \ref{rel-in C(D)} (i) and (ii), the solvability of 
$$u^2-\frac{D}{4}v^2=-1 \quad (4|D), \qquad u^2-Dv^2=-1\quad  (4\nmid D). \qquad (4.1)$$
Here for the later one, note that if $x^2-Dy^2=-4$ has solution then has solution with $x$ even. Let 

$$\psi=\matrixx{u}{\frac{D}{4}v}{v}{u}\quad  (4|D), \qquad \matrixx{u-v}{\frac{D-1}{2}v}{2v}{u+v}\quad  (4\nmid D). \qquad (4.2)$$
Then the determinant of $\psi$ is $-1$ by (4.1), so the torus bundle $M_\psi$ is non-orientable.
It is the direct calculation that $D_\psi=D$.
Then  $M=M_{\psi^2}$ is the unique orientable double cover of $M_\psi$.
Since $\psi^2$ and $\psi$ have the same fixed points as actions on $\mathbb R\cup \infty$, we have $D_{\psi^2}=D_\psi$ by definition.
So $D_M=D$.

On the other hand, assume (2), i.e. the double cover $M_{\psi^2}$ of $M_\psi$, $\det \psi=-1$, has discriminant $D$,  
then by the equality $$\psi^2=( w\psi)^{-1} w\psi^2 w (w\psi),\qquad (4.3)$$ 
where $w=\matrixx{0}{1}{1}{0}$ with $\det w=-1$. Since $\det w\psi=1$,  by (4.3) we have $$[Q_{\psi^2}]=[Q_{w\psi^2 w}].\qquad (4.4)$$

It is a direct verification that $$Q_{w\phi w}((x,y)w)=-Q_\phi(x,y).\qquad (4.4+)$$ So we have
$$-Q_{\psi^2}(x,y)=Q_{w\psi^2 w}(y,x)$$
which implies that 
$$[-Q_{\psi^2}]=-[Q_{w\psi^2 w}]\qquad (4.5)$$

Combing (4.4), (4.5) and  $[-Q]=-[Q]+[-Q_{0, D}]$ for any $[Q]\in C(D)$ (Lemma \ref {rel-in C(D)} (iii)), we have $[-Q_{0, D}]=0$, that is 
$[-Q_{0, D}]=[Q_{0, D}]$, which  is equivalent to the negative Pell equation $x^2-Dy^2=-4$ has integral solution (Lemma \ref {rel-in C(D)} (ii)). This shows the equivalence between (1) and (2).

Note that the double cover $M_\phi$ of $N_\psi$ is given as follows:
$$\phi=\matrixx{2bc+1}{2bd}{2ac}{2bc+1}, \qquad \text{if}\ \psi=\matrixx{a}{b}{c}{d}, \ abcd\neq 0.$$
It is clear that $M_\phi$ has discriminant $D_\phi=4\frac{abcd}{\gcd (ac, bd)^2}$. Note that
\begin{itemize}

\item $2[Q_\phi]\in C(D_\phi)$ is the zero class by Theorem \ref{Gauss}, 
and thus $M_\phi$ is achiral if and only if the negative equation $x^2-D_\phi y^2=-4$ has solution. 

\item $N_\psi$ is achiral if and only $-1$ is a self-mapping degree of $N_\psi$ by Lemma \ref{-1}, 
and  if and only if $a+d=0$ by \cite[Theorem 1.7]{SWW}.

\item if $a+d=0$, then $D_\phi=-\frac{4bc}{\gcd(b, c)^2}$, and therefore  
$-4=4a^2+4bc=4a^2-D_\phi \gcd (b, c)^2$.   

\item if if $N_\psi$ is achiral, then $M_\phi$ is achiral, since the cover $M_\phi\to N_\psi$ is characteristic by the "Moreover" part of Proposition \ref{Class1}
\end{itemize}
It follows that if $N_\psi$ is achiral with $D_\phi=\lcm (D, 4)$, then $x^2-Dy^2=-4$ has solutions. 

On the other hand, if the negative Pell equation $x^2-Dy^2=-4$ has solution, then there is always solution with $x$ even, say $(2a, b)$. Take 
$$\psi=\begin{cases}\matrixx{a}{b}{-Db/4}{-a}, \qquad &\text{if $4|D$}, \\
	\matrixx{a}{b/2}{-Db/2}{-a}, &\text{if $D\equiv 1\mod 4$}.
\end{cases}$$
Then $M_\phi$ is achiral of discriminant $D$ or $4D$. This shows the equivalence between (1) and (3). 	
\end{proof}

\begin{cor}\label{Pell2} Let $\CM$ be a commensurable class of Sol 3 manifolds of discriminant $D$. Then the following conditions are equivalent.
\begin{enumerate}
	\item the negative Pell equation $x^2-D y^2=-4$ has solutions;
	\item $\CM$ contains a non-orientable Sol 3-manifolds;
	\item $\CM$ contains an achiral tour semi-bundle.
\end{enumerate}
\end{cor}

The equation  $x^2-Dy^2=-4$ is called the negative Pell equation. The solvability of negative Pell equation has a long history and P. Stevenhagen \cite{Ste} made the following conjecture. 

\begin{conj}[Stevenhagen]   The density of fundamental positive discriminants $D$, for which $x^2-Dy^2=-4$ have solutions,  among all fundamental positive discriminants $D$ without prime factors $\equiv 3\mod 4$ is $1-\rho$ with
$$\rho:=\prod_{j=1}^\infty  \left(1+2^{-j}\right)^{-1} = 0.41942\cdots.$$\end{conj}
This conjecture is recently proved by P. Koymans and C. Pagano \cite{KP}.


\begin{lem}
$$\lim_{X\ra \infty}  \frac    {\#\{D_{\CM}<X| D_{\CM}\, \text{contains no prime $\equiv 3\mod 4$}\}}{X}=0,\qquad (4.6)$$
$$\lim_{X\ra \infty} \frac {\#\{D_{\CM}<X\}} {X}=C>0, \qquad (4.7)$$
\end{lem}

\begin{proof}
Note that a  integer $n>0$ is the sum of two integral squares if and only if the square-free part of $n$ contains no prime  $\equiv 3\mod 4$ \cite[Chap. 2]{Fl}. So we have 
$$\{D_{\CM}| D_{\CM}\, \text{contains no prime $\equiv 3\mod 4$}\}=\{D_{\CM}| D_{\CM}=a^2+b^2 \}. \qquad (4.8)$$

It is a classical fact that 

$$\lim_{X\ra \infty} \frac{\#\{0<D<X |D=a^2+b^2 \}}{X}=0, \qquad (4.9)$$
Then (4.6) follows from (4.8) and (4.9).






Let $S$ be the set of fundamental discriminants. Then for each $D\in S$,%
$D>0, \, D\equiv 0,1\mod 4,$  which is equivalent to 
$$D\equiv 1 \mod 4\,\,{\it or}\,\, D\equiv 0 \mod 8\,\,{\it or}\,\, D\equiv 4 \mod 8.$$
Therefore, the set $S$ equals the disjoint union of the following three subsets: 
$$S_1=\{ D\equiv 1 \mod 4\mid D\in S \}, \,S_2=\{ D\equiv 0 \mod 8\mid D\in S \}, \,S_3=\{ D\equiv 4 \mod 8\mid D\in S  \},$$
Now we can cite \cite [p12]{FK} to prove (4.7): 
 all densities of $S_1$, $S_2$ and $S_3$  exist and equal to $2/\pi^2$, $1/2\pi^2$, $1/2\pi^2$ respectively. So the density of $S$ exists and equals
$C=\frac{2}{\pi^2}+\frac{1}{2\pi^2}+\frac{1}{2\pi^2}=\frac{3}{\pi^2}>0.$
\end{proof}
\begin{proof}[Proof of Theorem \ref{density}] 
 (1) 
 Let $D$ be prime  in the form of $4k+1$. The $D=D_\CM$ for some commensurable class $\CM$ by Proposition \ref{Class2}.
Then $\CM$ is achiral by the "only if" part of  Corollary \ref{Gauss2}. It is a well known fact that there are infinitely many primes in this form of 
$4k+1$ \cite[Theorem 10.5]{EW}.

Now we prove the "however" part. (4,7) and (4.8) implies that

$$\lim_{X\ra \infty} \frac    {\#\{D_{\CM}<X| D_{\CM}\, \text{contains no prime $\equiv 3\mod 4$}\}}    {\#\{D_{\CM}<X\}}=0.\qquad (4.9)$$

By Theorem \ref{Class2}, we have
$${\#\{\CM\, |\, D_{\CM}<X\}}= {\#\{D_{\CM}<X\}}\qquad (4.10)$$

By Theorem \ref{Class2} and Theorem \ref{Gauss2}, we have 
$${\#\{\text{achiral} \ \CM\  |D_{\CM}<X\}}={\#\{ D_{\CM}<X | \CM  \ \text{achiral}\ \}}$$$$= \#\{D_{\CM}<X| D_{\CM}\, \text{contains no prime $\equiv 3\mod 4$}\}\qquad (4.11)$$
Then (4.10), (4.11) and (4.9) imply 
$$\lim_{X\ra \infty} \frac{\#\{\text{achiral}\ \CM\ \ |D_{\CM}<X\}}{\#\{\CM \ |\ D_{\CM}<X \ {}\}}=0,$$

(2) By Theorem \ref{Class2} and Corollary \ref{Pell2}, we have 
 $$\#\{ \CM \, \text{contains an non-orientable element }\, |\, D_\CM <X\ \}$$
 $$=\#\{D_\CM <X \,|\, \CM \, \text{contains an non-orientable element}  \}$$ 
 $$=\#\{D_\CM <X \,|\, \text{$x^2-D_\CM y^2=-4$ has solutions}  \}.\qquad(4.12)$$

(4.11) and (4.12) and  the positive answer of Stevenhagen Conjecture \cite{KP} imply
   $$\lim_{X\ra \infty} \frac {\#\{ \CM \ {}\text{contains an non-orientable element \,}|\, D_\CM <X\ \}}{\#\{\text{achiral}\ \CM\ \ |D_{\CM}<X\}}=1-\rho.$$
\end{proof}



\begin{prop}\label{Gauss4} There exists $M\in \Sol_D$ double cover a tours semi-bundle if and only if $4|D$. Moreover,  $M_\phi \in \Sol_D$ is the double cover of a tours semi-bundle if and only if $2[Q_\phi]=0$ and $\phi\equiv I \mod 2\BZ$, in this case, there is a unique tours semi-bundle double covered by $M_\phi$.  
\end{prop}

\begin{proof} By \cite{Ma}, $M_\phi\in \Sol_D$ is the double of a tours semi-bundle if and only if there exists $A\in \GL_2(\BZ)$ such that 
$$A\phi=\phi^{-1} A, \qquad \det A=-1, \qquad A\equiv \phi \equiv I \mod 2\BZ.$$
Let $A=\matrixx{x}{y}{z}{w}$ and assume that $Q_\phi=\alpha x^2+\beta xy +\gamma y^2$. Then the condition $A\phi=\phi^{-1}A$ and $\det A=-1$ is equivalent to
$$A=\matrixx{x}{y}{z}{-x}, \quad x^2+yz=1, \quad \alpha z=\beta x+\gamma y.$$
One can check that
$$\matrixx{-x}{y}{-z}{-x}\matrixx{\alpha}{-\beta/2}{-\beta/2}{\gamma}\matrixx{-x}{-z}{y}{-x}=\matrixx{\alpha}{\beta/2}{\beta/2}{\gamma}.$$
It follows that the condition that there exists $A\in \GL_2(\BZ)$ such that $A\phi=\phi^{-1}A$ and $\det A=-1$ if and only if $2[Q_\phi]=0$. 

Note that $\phi\equiv I\mod 2$ implies $4|D$.  Assume that $4|D$ take $M_\phi\in \Sol_D$ such that $2[Q_\phi]=0$, we have that $\phi$ is  $\SL_2(\BZ)$-conjugate to $\phi_0=\matrixx{a}{b}{c}{a}$. It is easy to see that $A_0:=\matrixx{1}{0}{0}{-1}\equiv I \mod 2$ satisfies $A_0\phi_0=\phi_0^{-1}A_0$, $\det A_0=-1$, and that  $\phi_0^2\equiv I\mod 2$. Thus $M_{\phi^2}\in \Sol_D$ is the double cover of a tours semi-bundle. 

Let $A, A'$ be two such matrices, then $A'=AC$ with $C\in \SL_2(\BZ), \equiv I\mod 2$ and $C\phi=\phi C$. One can easily check that $A'$ is equivalent to $A$ in the sense of \cite{Ma}. Thus if $M_\phi$ is the double cover of at most one tours semi-bundle.  	
\end{proof}

 \section{Shimizu L-functions}
 
 In the 1970s, Hirzebruch \cite{Hir} formulated a conjecture on the signature of the Hilbert modular varieties whose cusp cross-sections are solvmanifolds, which relates the signature defects of these Sol manifolds to special values of Shimizu's L-functions of totally real fields.  We now recall the definition of Shimizu L-functions for Hilbert modular surfaces. In this case,  these solvmanifolds are oriented tours bundle $M_\phi$. 
 
 Given $\phi=\begin{pmatrix}a&b\\c&d\end{pmatrix}\in\SL_2(\BZ)$ with discrimimant $D_\phi=D$ and quadratic form $Q_\phi$:
	
	the Shimizu L-function is defined as 
	\[L(M, s):=\sum_{(0, 0)\neq (x,y)\in \BZ^2/\phi^\BZ}\frac{\sign\ Q_\phi(x,y)}{|Q_\phi(x,y)|^s}, \qquad \Re(s)>1\]
	only depends on its oriented equivalence class. In fact $L(M, s)$ has an entire continuation to $\BC$ (See Proposition \ref{entire}).

\begin{thm}\label{Shimizu}
	Let $M=M_\phi$ be an oriented Sol torus bundle. 
	
	(1) $M$ is achiral if and only if
		 the Shimizu L-function $L(M, s)\equiv 0$.
	
(2) If $L(M_{\phi_1}, s)=L(M_{\phi_2}, s)$ for  two non-achiral oriented bundles $M_{\phi_1}, M_{\phi_2}$ with same discriminant, then $M_{\phi_1}$ is homeomorphic to either $M_{\phi_2}$ or $M_{-\phi_2}$ as oriented torus bundle.  
\end{thm}

We list some facts for the proof of Theorem \ref{Shimizu}.

\begin{prop} \cite{Web}\label{Web} Each $[Q]\in C(D)$ represents infinitely many primes. 
\end{prop}

\begin{lem}\label{represent prime}
It both $Q_1, Q_2\in C(D)$ represent a prime $p$ with $(p, 2D)=1$, then $[Q_1]=\pm [Q_2]$.
\end{lem} 
\begin{proof}
By \cite[Exercise 3.30]{EW}, each $Q_i$ is $SL_2(\BZ)$-equivalent to a form $Q_i'(x,y)
=px^2+d_ixy+e_iy^2$, $d_i\le p$. Note $d_i-4pe_i=D$. Clearly $d_1$ and $d_2$ has the same parity.
Therefore $\frac{|d_1|-|d_2|}2, \frac{|d_1|+|d_2|}2\in\mathbb{Z}$. By $p \nmid D$, we have $p \nmid d_i$ and therefore $0<|d_i|<p$.
Therefore $0<\frac{|d_1|+|d_2|}2<p$. Since
$$\frac{d_1^2-d_2^2}4=\frac{|d_1|-|d_2|}{2}\frac{|d_1|+|d_2|}{2}=p(e_1-e_2).$$
Therefore $p|\frac{|d_1|-|d_2|}2.$ 
It follows  $|d_1|=|d_2|, e_1=e_2$. So  $[Q_1']=\pm [Q_2']$ and  $[Q_1]=\pm [Q_2]$.
\end{proof}

\begin{lem}\label{Q+tr} Suppose $\phi, \psi\in \SL_2(\BZ)$ with $D_\phi=D_\psi=D$.
If $[Q_{\phi}]=\pm [Q_{\psi}]$ and $\tr(\phi)=\tr(\psi)$.  Then $\phi$ is $\SL_2(\BZ)$-conjugate to either $\psi$ or $w\psi^{-1}w$.
\end{lem}
\begin{proof} (i) First note that $\phi$ is determined by $Q_\phi$ and $\tr(\phi)$: Suppose $Q_\phi(x,y)=\alpha x^2+\beta xy + \gamma y^2$, one can solve that
 $\phi=\sign\tr(\phi)\matrixx{\frac{t+\beta s}{2}}{-\gamma s}{\alpha s}{\frac{t-\beta s}{2}}$, 
 where $t=|\tr(\phi)|$ and $s=\sqrt{\frac{t^2-4}D}$.
 
 (ii) If $[Q_{\phi}]=[Q_{\psi}]$, then  $Q_{\phi}(x,y)=Q_{\psi}((x,y)A)$; If $[Q_{\phi}]=-[Q_{\psi}]$, then  $Q_{\phi}(x,y)=Q_{\psi}((x,y)A'w)$;
 both $A, A'\in \SL_2(\BZ)$. Then one can verifies the Lemma by (i), (3.3) and Lemma \ref{Q and -Q}.
 \end{proof}

By Lemma \ref{Q and -Q}  we have 	
\begin{lem}	\label{II}
 $L(M_{-\phi}, s)= L(M_\phi, s)$, $L(M_{\phi^{-1}}, s)= -L(M_\phi, s)$ and  $L(M_{\phi}, s)= -L(M_{w\phi w}, s)$. 
\end{lem}


\begin{proof}[Proof of Theorem \ref{Shimizu}]
Suppose $M$ is achiral, then $M_\phi=-M_\phi$, and furthermore  $M_\phi=M_{\phi^{-1}}$ by Lemma \ref{inverse}. 
Then  $\phi^{-1}$ is $SL_2(\BZ)$ conjugate to either $\phi$ or $w\phi^{-1} w$ by Lemma \ref{oriented}. That is either
$$ L(M_{\phi^{-1}}, s)= L(M_\phi, s) \ \ \text{or} \ \ L(M_{\phi^{-1}}, s)= L(M_{w\phi^{-1} w}, s)\qquad (5.1)$$
Combining (5.1) and  Lemma \ref{II}, we have  $L(M, s)\equiv 0$. 

On the other hand, let 
$K_{n, \phi} =\{(x, y)\in \BZ\oplus \BZ |Q(x, y)=n\}/\phi$ for each $n\in\BZ$. $K_{n, \phi}$ is finite. 
In $L(M_\phi, s)$, the term $1/n^s$ occurs $|K_{n, \phi}|$ times since each orbit in $K_{n, \phi}$ contribute a single $1/n^s$.
Similarly the term $-1/n^s$ occurs $|K_{-n, \phi}|$ times. So we have 

$$L(M_\phi, s)= \sum_{n=1}^{\infty}\frac {-|K_{-n, \phi}|+|K_{n, \phi}|}{|n|^s}.\qquad(5.2)$$

Hence  $L(M_\phi, s)\equiv 0$ if and only if  $|K_{-n, \phi}|=|K_{n, \phi}|$ for each $n$, which is equivalent to 
$Q$ and $-Q$ represent the same set of integers. Then both $Q$ and $-Q$ represent some prime $p$ with $(p, 2D)=1$ by Proposition \ref{Web}, and then 
  $[-Q_\phi]=\pm [Q_\phi]$ by Lemma \ref{represent prime}.
By Theorem \ref{main1}, this is equivalent to that $M$ is achiral.

We have proved (1). Below we prove (2).

Suppose both $\tr(\phi_1), \tr(\phi_2)>0$ (otherwise replace $\phi$ by $-\phi$).
Let $\psi_1=\phi_1^{n_1}$ and  $\psi_2=\phi_2^{n_2}$ such that  $\tr {\psi_1}=\tr {\psi_2}$, where $n_i$ a positive integer, $i=1,2$.
Then from the definition, it is easy to see $$L(M_{\psi_i}, s)=n_iL(M_{\phi_i},s),  i=1, 2. \qquad (5.3)$$
Following (5.2) and  $L(M_{\phi_1}, s)=L(M_{\phi_2}, s)$ we have 
$$L(M_{\psi_1}, s)=\frac {n_2}{n_1}L(M_{\psi_2}, s)=\sum_{n=1}^{\infty}\frac {-|K_{-n}|+|K_{n}|}{|n|^s} \qquad (5.4)$$

By Proposition \ref{Web}, $Q_{\psi_1}$ represents a prime $p$ with $(p, 2D)=1$. Then $-p$ is not represented by $Q_{\psi_1}$ otherwise $[Q_{\psi_1}]=\pm[-Q_{\psi_1}]$  by Lemma \ref{represent prime},
  and therefore $M_{\psi_1}$ is achiral
by Theorem \ref{main1}. 
 So we have $K_p\ne 0$ and $K_{-p}=0$. By (5.4)  $Q_{\psi_2}$ also represents $p$.
Then $[Q_{\psi_2}]=\pm [Q_{\psi_1}]$ again by Lemma \ref{represent prime}. 
Since $\tr {\psi_1}=\tr {\psi_2}$,
 $\psi_1$ is $SL_2(\BZ)$-conjugated to either $\psi_2$ or $w\psi_2^{-1}w$ by Lemma \ref{Q+tr}.
Therefore $L(\psi_1,s)=L(\psi_2,s)$ by Lemma \ref{II}.
That is to say $n_1=n_2$, and therefore $\phi_1$ is $SL_2(\BZ)$-conjugated to either $\phi_2$ or $w\phi_2^{-1}w$.
 Since $M_{\phi_2}$ and $M_{w\phi_2^{-1}w}$ are homeomorphic
as oriented torus bundles, in either case, $M_{\phi_1}$ and  $M_{\phi_2}$  are homeomorphic
as oriented torus bundles.

When  requiring $tr(\phi_i)>0$, we replace $\phi_i$ by $-\phi_i$, so in  general we have $M_{\phi_1}$ is homeomorphic to either $M_{\phi_2}$ or $M_{-\phi_2}$ as oriented torus bundles.
\end{proof}
	
	\begin{prop}\label{entire} 	For an oriented tours bundle $M=M_\phi$ of discriminant $D$, 
	$L(M, s)$ has an entire continuation to $\BC$ and satisfies a functional equation:
\[\Lambda(M,s):=\Gamma_\BR(s+1)^2D^{s/2} L(M,s)=\Lambda(M,1-s).\]	
\end{prop}

\begin{proof} It follows (see \cite{ADS}) that $L(M, s)$ ha s entire continuation and satisfies the functional equation: 
$$\Lambda(M, s)=-\Lambda(M_{\phi^t}, 1-s).$$ 
Since for any $\phi\in \SL_2(\BZ)$, $\phi^t$ is $\SL_2(\BZ)$ conjugate to $\phi^{-1}$:
	$$\phi^t=\matrixx{0}{1}{-1}{0} \phi^{-1} \matrixx{0}{1}{-1}{0}.$$
 By Lemma \ref{II}, we have  $L(M_{\phi^t}, s)=L(M_{\phi^{-1}}, s)=-L(M, s)$. So $\Lambda(M, s)=\Lambda(M, 1-s)$. 
 \end{proof}


\section{Appendix: Proof of Lemma 3.3}

 {\bf Lemma 3.3}  {\it For  forms of discirminant $D$,
  
 (i) $[Q]=[Q_{0, D}]$ is equivalent to $Q$ represents $1$, $[Q]=[-Q_{0, D}]$ is equivalent to $Q$ represents $-1$;
 
 (ii) $[-Q_{0, D}]=[Q_{0, D}]$ is equivalent to $x^2-Dy^2=-4$ has an integral solution;
 
 (iii) $[-Q]=-[Q]+[-Q_{0, D}]$; 
 
 (iv) $[-Q]=\pm [Q]$ is equivalent to either $[Q_{0, D}]=[-Q_{0, D}]$ or $2[Q]=[-Q_{0, D}]$;
 
(v)  $[Q_1]=-[Q_2]$  if and only if $Q_1(x, y)=Q_2((x, y)B)$, where $B\in \GL_2(\BZ)$ and $\det B=-1$.}

\begin{proof}Dirichlet gave an approach to the composition law of $C(D)$. Let $[\alpha_i, \beta_i, \gamma_i]$, $i=1, 2$, be two primitive forms of discriminant $D$ satisfying $\gcd(\alpha_1, \alpha_2, (\beta_1+\beta_2)/2)=1$. Let $\beta$ be an integer (unique mod $2\alpha_1\alpha_2$) such that 
$$\beta\equiv \beta_1\mod 2\alpha_1, \quad \beta\equiv \beta_2\mod 2\alpha_2, \quad \beta^2\equiv D \mod 4\alpha_1\alpha_2.$$
Then the composition 
$$[\alpha_1, \beta_1, \gamma_1]+[\alpha_2, \beta_2, \gamma_2]=[\alpha_1\alpha_2, \beta, (\beta^2-D)/4\alpha_1\alpha_2].\qquad (2.1)$$
In fact, any two classes have concordant forms, which means  

$$(i) \, \alpha_1\alpha_2\neq 0, \ \ (ii)\,  \beta_1=\beta_2, \ \  (iii) \, \alpha_2|\gamma_1, \alpha_1|\gamma_2\qquad (2.2)$$. 

In this case, $\beta=\beta_i$. Moreover, we may even require that $\gcd(\alpha_1, \alpha_2)=1$ and $\gcd(\alpha_1\alpha_2, u)=1$ for any given $u\neq 0$ (See Chapter 5 Lemma 2.5 in \cite{Fl}).  It is easy to see that 
$$-[\alpha, \beta, \gamma]=[\gamma, \beta, \alpha]=[\alpha, -\beta, \gamma]. \qquad (2.3)$$ 

Recall $Q_{0, D}$ is the principal form $x^2-\frac{D y^2}4$ if $4|D$ and $x^2+xy+\frac{1-D}{4}y^2$ otherwise. 
In the proof of (i) and (ii), we just verify the case $4|D$, and  the verification for $D=4k+1$ is similar.

(i) It is clear that $[Q]=[Q_{0, D}]$ if and only if $Q$ representing $1$ is  equivalently 
that $[Q]=[-Q_{0, D}]$ if and only if  $Q$ representing $-1$.
 We just prove the former:  Clearly $Q_{0, D}$ represents 1. On the other hand, Suppose $Q(x_0, y_0)=1$ for some integers $x_0,y_0$.
 Then $(x_0, y_0)=1$, and $(x_0, y_0)=(1,0)A$ for some $A\in SL_2\BZ$ by Bezout's Theorem. Then for 
 $Q'(x,y)=Q((x,y)A)$, $Q'(1,0)=1$, which implies that $Q'(x,y)=x^2+bxy+cy^2$. If $4|D$, $b$ is even,
 then $Q'(x,y)=(x-\frac b2y)^2-\frac D4 y^2$.


(ii) Suppose $[-Q_{0,D}]=[Q_{0,D}]$. By (i), when $4|D$, $Q_{0,D}(x_0,y_0)=x_0^2-\frac{D}{4}y_0^2=-1$  for some integers $x_0,y_0$, 
so $(2x_0)^2-Dy_0^2=-4$; So $x^2-Dy^2=-4$ has an integral solution. 


Suppose $x^2-Dy^2=-4$ has an integral solution $(x_0,y_0)$. If $4|D$, then $0\equiv-4=x_0^2-Dy_0^2\equiv x_0^2 ({\rm mod\,\,} 4)$. So $2|x_0$,  
and $Q_{0,D}(\frac{x_0}{2},y_0)=\frac{x_0^2}{4}-\frac{D}{4}y_0^2=\frac{x_0^2-Dy_0^2}{4}=-1$.
By (i), $[-Q_{0,D}]=[Q_{0,D}]$.

(iii) For any $[Q]=[\alpha, \beta, \gamma]$. Then
$$[-\alpha, -\beta, -\gamma]+[\alpha, \beta, \gamma]=[-\alpha, -\beta, -\gamma]+[\gamma, -\beta, \alpha]=[-\alpha \gamma, -\beta, -1]=[-Q_{0, D}].\qquad (2.4)$$
The secand, the third, and last equality are  based on (2.3), (2.1) and (i). repectively.

(iv)   If $[-Q]=-[Q]$, then by (iii) we have $-[Q]=-[Q]+[-Q_{0, D}]$, that is  $[Q_{0, D}]=[-Q_{0, D}]$.
If  $[-Q]=[Q]$, then by (iii) we have $[Q]=-[Q]+[-Q_{0, D}]$, that is  $2[Q]=[-Q_{0, D}]$.

(v)  Suppose $Q_1(x, y)=Q_2((x, y)B)$, where $B\in \GL_2(\BZ)$ and $\det B=-1$.
Let $w=\matrixx{}{1}{1}{}$. Since $w^2B=B$ and $wB\in \SL_2(\BZ)$, we may assume that $B=w$.
If  $Q_2=[\alpha, \beta,\gamma]$, then $Q_1=[\gamma, \beta, \alpha]$. 
Then by (2.3) $[Q_1]=-[Q_2]$.

Suppose  $[Q_1]=-[Q_2]$. Let $Q'_1(x, y)=Q_2((x, y)B')$, where $B'\in \GL_2(\BZ)$ and $\det B=-1$. Then we just proved that
$[Q'_1]=-[Q_2]$. So $[Q_1]=[Q'_1]$. The remaining follows from the definition.
\end{proof}


\begin{thebibliography}{}


\bibitem[ADS]{ADS} M.F.Atiyah, H.Donnelly, I.M.Singer, {\em Eta invariants, signature defects of cusps, and values of L-functions}, Annals of Math. 118 (1983), 131-177.


\bibitem[BGV]{BGV} A. P. Bareto, D. L. Goncalves and D. Vedruscolo, {\em Free involutions on torus semi-bundles and the Borsuk-Ulam theorem for maps into $\BR^n$}, Hiroshima Math. J. \textbf{46} (2016), 255-270.



\bibitem[BS]{BS} W. Bosma and P. Stevenhagen, {\em On the computation of quadratic 2-class groups}, Journal de Theorie des Nombres de Bordeaux, tome 8, $n^o 2$ (1996), p.283-313.

\bibitem[EW]{EW} G. Everest, T. Ward,  {\it An introduction to number theory.} Graduate Texts in Mathematics, 232. Springer-Verlag London, Ltd., London, 2005. 








\bibitem[Fl]{Fl} D. Flath, {\em Introduction to Number Theory}, Wiley, New York, 1988. 

\bibitem[Gauss]{Gauss} C.F. Gauss, {\em Disquisitiones Arithmeticae}, Gerhard Fleischer, Leipzig, 1801.



\bibitem[Ha]{Ha} A. Hatcher, {\it Notes on basic $3$-manifold topology.} \url{http://pi.math.cornell.edu/~hatcher/}.


\bibitem[He1]{He1} J. Hempel, {\it $3$-manifolds,} Princeton University Press and University
of Tokyo Press, 1976.

\bibitem[He2]{He2} J. Hempel, 
{\it Residual finiteness for 3-manifolds. Combinatorial group theory and topology (Alta, Utah, 1984), 379-396, 
Ann. of Math. Stud., 111, Princeton Univ. Press, Princeton, NJ, 1987.}

\bibitem[Hr]{Hir} F. Hirzebruch, {\em Hilbert modular surfaces}. L'Enseign. Math. 19 (1973), 183-281. 



\bibitem[KP]{KP} P. Koymans and C. Pagano, {\em On Stevenhagen's conjecture}, arXiv: 2201.13424v1 [math.NT] 31 Jan 2022. 


\bibitem[LS]{LS} Y. Liu and H. Sun, {\it Toward and after virtual specialization in 3-manifold topology.} Surveys in differential geometry 2020. Surveys in 3-manifold topology and geometry, 215-252, Surv. Differ. Geom., 25, Int. Press, Boston, MA.

\bibitem[Ma]{Ma} B. Martelli
{\it An Introduction to Geometric Topology}
arXiv:1610.02592   math.GT math.DG (2016)


\bibitem[Sak]{Sak}M. Sakuma, {\em Involutions on torus bundles over $S^1$}, Osaka J. Math. \textbf{22} (1985), 163-185.



\bibitem[Sc]{Sc} P. Scott, {\sl The geometries of 3-manifolds}, Bull. London Math. Soc. {\bf 15} (1983), 401--487.



\bibitem[Sh]{Sh} H. Shimizu, {\it On discontinuous groups operating on the product of the upper half planes.}  Ann. of Math. (2) 77 (1963), 33-71.

\bibitem[Smi2]{Smith2} A. Smith.
\emph{$2^\infty$-Selmer groups, $2^\infty$-class groups,
and Goldfeld's conjecture}.
Preprint.

\bibitem[Smi3]{Smith3} A. Smith,  \emph{$\ell^\infty$-Selmer groups
in degree $\ell$ twist families}. Ph. D. Thesis.

\bibitem[Ste]{Ste} P. Stevenhagen, {\em The number of real quadratic fields having units of negative norm}. Experiment. Math. 2:121-136, 1993. 


\bibitem[SWW]{SWW}
H. Sun, S. Wang, J. Wu, {\em Self-mapping degrees of trus bundles and torus semi-bundles}, Osaka J. Math. {\bf 47} (2010), 131--155.



\bibitem[Th]{Th} W. Thurston,  {\it Three-dimensional manifolds, Kleinian groups and
hyperbolic geometry}, Bull. Am. Math. Soc. 6 (1982), 357--381.


\bibitem[Th1]{Th1} {W.~P.~Thurston}, {\it Three-dimensional geometry and topology.} Vol. 1. Edited by Silvio Levy. Princeton Mathematical Series, 35. Princeton University Press, Princeton, NJ, 1997.

\bibitem[Web]{Web} H. Weber, {\it Beweis des Satzes.} Math. Ann. 20 (1882), 301-329.

\bibitem[Zhang]{W. Zhang}{W.P. Zhang},  {\em Lectures on Chern-Weil Theory and Witten Deformations}, Nankai Tracts in Mathematics Vol. 4, World Scientific. 




\end{thebibliography}
\end{document}